\documentclass[10pt,a4paper]{scrartcl} % KOMA-Script article 
\usepackage[nochapters%,eulermath
]{classicthesis} % nochapters
\pdfoutput=1

\usepackage[left=1.5in,right=1.5in,top=1in,bottom=1in]{geometry}

\usepackage{amsmath}
\usepackage{amssymb}
\usepackage{amsthm}
\usepackage{graphicx}
\usepackage[mathscr]{eucal} 
\usepackage{mathrsfs}

\usepackage{algorithm2e}

\theoremstyle{plain}

\newtheorem{lemma}{Lemma}
 
\theoremstyle{definition}
\newtheorem*{definition}{Definition}

\newtheorem{example}{Example}

\def \cC {\mathcal{C}}
\def \cG {\mathcal{G}}
\def \cP {\mathcal{P}}

\def \FF {\mathbb{F}}

\def\Plus{+}

\newcommand{\AND}{\mathbin{\textup{\texttt{\&}}}}

\DeclareMathOperator{\Span}{span}
\DeclareMathOperator{\Length}{length}

\begin{document}

\title{\rmfamily\normalfont\textsc{(Re)constructing Code Loops}}
\author{\textsc{Ben Nagy and David Michael Roberts}}
\date{10 March 2020}

\maketitle

\begin{abstract}
The Moufang loop named for Richard Parker is a central extension of the extended binary Golay code. 
It the prototypical example of a general class of nonassociative structures known today as \emph{code loops}, which have been studied from a number of different algebraic and combinatorial perspectives.
This expository article aims to highlight an experimental approach to computing in code loops, by a combination of a small amount of precomputed information and making use of the rich identities that code loops' twisted cocycles satisfy.
As a byproduct we demonstrate that one can reconstruct the multiplication in Parker's loop from a mere fragment of its twisted cocycle. 
We also give relatively large subspaces of the Golay code over which Parker's loop splits as a direct product.
\end{abstract}

\section{Introduction.}

Associativity is one of the standard axioms for groups, along with inverses and the identity element.
At first it can be difficult to imagine what algebraic structures could be like that are \emph{not} associative. 
However, we are all familiar with the nonassociative operation of exponentiation:  $(2^3)^{\!{}^2} \neq 2^{(3^2)}$, for example.
But exponentiation is an ill-behaved binary operation, with neither (a two-sided) inverse nor identity element, hence very much unlike a group. 
One class of nonassociative structures as close to being a group as possible are \emph{Moufang loops}, which are more or less ``groups without full associativity'', having an identity element and two-sided inverses.
This article is concerned with a special class of Moufang loops now known as \emph{code loops}, which have curious connections to sporadic finite simple groups \cite{Conway}, \cite[\S 7]{Griess87}.
In particular, we consider in detail a Moufang loop $\cP$, introduced by Richard Parker, in unpublished work in the early 1980s, as the exemplar of a class of loops later shown to be equivalent to code loops \cite{Griess}.

An early example of a code loop, though not identified as such, appeared in Coxeter's work on integral octonions \cite{Coxeter} (see, e.g., \cite[\S 2]{Griess87} for details), but the most striking appearance was the use of Parker's loop $\cP$ as part of John Conway's construction of the Monster sporadic simple group \cite{Conway}. 
A general study of loops of this type was subsequently made by Robert Griess \cite{Griess}, who named them, based on the fact that they can be built from doubly-even binary codes. 
For example, Parker's loop $\cP$ is then constructed by starting with the famous extended Golay code.
Griess also proved the existence of code loops by an algorithmic construction, which we adopt below.
(More recent approaches will be discussed in Section~\ref{sec:discussion}.)

The data required to describe a code loop is a function $\theta\colon \cC\times \cC \to \FF_2$, where $\FF_2 = \{0,1\}$ is the field with two elements, and $\cC\subset (\FF_2)^n$ is a doubly-even binary code---a subspace with a particular property that we recall below. 
The function $\theta$ must then satisfy a number of identities that use this property.
For the sake of terminology, we refer to the elements of $\FF_2$ as \emph{bits}, and $\theta$ is called a \emph{twisted cocycle}.
Note, however, that both Conway's and Griess's treatment of Parker's loop is rather implicit, merely using the properties of $\theta$, not examining its structure. 
The point of the present article is to give a concrete and efficient approach to code loops in general, and Parker's loop in particular---the latter especially so, given its use in constructing the Monster group.

Recall that the elements (or \emph{words}) in a code $\cC$, being vectors, can be combined by addition---this is a group operation and hence associative. 
The elements of a code loop consist of a pair: a code word and one extra bit.
The extra bit, together with $\theta$, \emph{twists} the addition so that the binary operation $\star$ in the code loop is a \emph{nonassociative operation}: $(x\star y)\star z\not=x\star (y\star z)$ in general.

While addition of words in a code is performed by coordinate-wise addition in $\FF_2$ (that is, bitwise XOR), the algebraic operation in a code loop is not easily described in such an explicit way---unless one has complete knowledge of the function $\theta$.
It is the computation and presentation of such twisted cocycles that will mainly concern us in this article, using Griess's method \cite[proof of Theorem 10]{Griess}.
As a result, we will observe some curious features of Parker's loop, obtained via experimentation and, it seems, previously unknown.

The function $\theta$ can be thought of as a table of bits, or even as a square image built from black and white pixels, with rows and columns indexed by elements of $\cC$. 
One of the main results given here is that $\theta$ can be reconstructed from a much smaller subset of its values. 
If $\dim\cC = 2k$, then while $\theta$ is a table of $(2^{2k})^2$ values, we need only store roughly $2^{2k}$ values, and the rest can be reconstructed from equation (\ref{eq:theformula}) in Lemma~\ref{lemma:formula lemma} below. 
For Parker's loop this means that instead of storing a multiplication table with approximately $67$ million entries, one only need store a $128\times 128$ table of bits. 
This table, given in Figure~\ref{fig:Parker cocycle}, exhibits several structural regularities that simplify matters further.

We begin the article with a treatment that (re)constructs a small nonabelian finite \emph{group} using a cocycle on a 2-dimensional $\FF_2$-vector space.
This will exhibit the technique we will use to construct code loops, in the setting of undergraduate algebra.

\section{Extensions and cocycles.}

Recall that the \emph{quaternion group} $Q_8$, usually introduced in a first course in group theory, is the multiplicative group consisting of the positive and negative basis quaternions:
\[
	Q_8 = \{1,\, i,\, j,\, k,\,-1,\, -i,\, -j,\, -k\}.
\]
The elements of $Q_8$ satisfy the identities
\[
	i^2 = j^2 = k^2 = -1, \quad ij=k.
\]
There is a surjective group homomorphism $\pi\colon Q_8 \to \FF_2 \times \FF_2 =: V_4$ (the Klein 4-group), sending $i$ to $(1,0)$ and $j$ to $(0,1)$, and the kernel of $\pi$ is the subgroup $(\{1,-1\},\times)\simeq (\FF_2,+)$.
Moreover, this kernel is the \emph{center} of $Q_8$, the set of all elements that commute with every other element of the group. This makes $Q_8$ an example of a \emph{central extension}: $\FF_2\to Q_8 \to V_4$.

Now $Q_8$ is a nonabelian group, but both $\FF_2$ and $V_4$ are abelian.
One might think that it shouldn't be possible to reconstruct $Q_8$ from the latter two groups, but it is! 
That is, if we are given some extra information that uses only the two abelian groups.
There is a function $s\colon V_4 \to Q_8$, called a \emph{section}, defined by 
\[
(0,0) \mapsto 1\qquad
(1,0) \mapsto i\qquad
(0,1) \mapsto j\qquad
(1,1) \mapsto k.
\]
This \emph{almost} looks like a group homomorphism, but it is not, as $(1,0) + (1,0) = (0,0)$ in $V$, but $s(1,0)s(1,0) = i^2 \not= 1 = s(0,0)$ in $Q_8$.
We can actually measure the failure of $s$ to be a group homomorphism by considering the two-variable function
\[
	d\colon V_4 \times V_4 \to \FF_2
\]
defined by $ (-1)^{d(v,w)} = s(v)s(w)s(v+w)^{-1}$. 
This function is called a \emph{cocycle}, or more properly a 2-cocycle (for a treatment of the general theory, see, e.g., \cite{Brown}, and Chapter IV in particular).
It is a nice exercise to see that $s(v)s(w)s(v+w)^{-1}$ is always $\pm 1$, so that this definition makes sense. The values of $d(v,w)$ are given as

\begin{center}
\begin{tabular*}{0.35\textwidth}{c|cccc}
$v\setminus w$&$00$&$10$&$01$&$11$\\
\hline
	$00$		& $0$& $0$& $0$& $0$\\
	$10$		& $0$& $1$& $1$& $0$\\
	$01$		& $0$& $0$& $1$& $1$\\
	$11$		& $0$& $1$& $0$& $1$\\
\end{tabular*}
\end{center}
where we write $00$ for $(0,0)$, $10$ for $(1,0)$, and so on.
Clearly, if $s$ \emph{were} a homomorphism, $d$ would be the zero function.
One can check that $d$ satisfies the \emph{cocycle identities}
\begin{equation}\label{eq:cocycle_eqn}
	d(v,w)-d(u+v,w)+d(u,v+w)-d(u,v) = 0
\end{equation}
for all triples $u,v,w\in V_4$. 
It is also immediate from the definition that $d(0,0)=0$.
An alternative visualization is given in Figure~\ref{fig:cocycle for q8}.

\begin{figure}[!ht]
\begin{center}
\includegraphics[height=2.5cm]{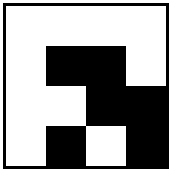}
\end{center}
\caption{A $4\times4$ array giving the values of the cocycle $d\colon V_4\times V_4\to \FF_2$, with white = $0$, black = $1$. The order of the row/column labels is $00$, $10$, $01$, $11$.}
\label{fig:cocycle for q8}
\end{figure}

The reason for this somewhat mysterious construction is that we can build a bijection of \emph{sets} using $s$ and the group isomorphism $\FF_2\simeq \{1,-1\}$, namely the composite
\[
	\phi\colon \FF_2\times V_4 \simeq \left(\{1\}\times V_4\right) \cup \left(\{-1\} \times V_4\right)\! \to\!
	\{1,\, i,\, j,\, k\}\cup \{-1,\, -i,\, -j,\, -k\} = Q_8\,.
\]
Now, if we define a new product operation using the cocycle $d$ on the underlying set of $\FF_2\times V_4$ by
\[
	(s,v)\ast_d(t,w):=(s+ t+ d(v,w),v+w),
\]
then the cocycle identities ensure that this is in fact associative and further, a group operation.
Finally, $\phi$ can be checked to be a homomorphism for multiplication in $Q_8$ and for $\ast_d$, hence is a group isomorphism.

Thus we can reconstruct, at least up to isomorphism, the nonabelian group $Q_8$ from the two abelian groups $V_4$ and $\FF_2$, together with the cocycle
\[
	d\colon V_4\times V_4\to \FF_2.
\]
If we didn't know about the group structure of $Q_8$ already, we could construct it from scratch using $d$.
We can also build loops (particularly code loops) in the same way, which we will now outline.

\section{Twisted cocycles and loops.}

The construction in the previous section is a fairly typical case of reconstructing a central extension from a cocycle (although in general one does not even need the analog of the group $V_4$ to be abelian). 
However, we wish to go one step further, and construct a structure with a \emph{nonassociative} product from a pair of abelian groups: the group $\FF_2$ and additive group of a vector space $V$ over $\FF_2$.
Instead of a cocycle, we use a \emph{twisted} cocycle: a function $\alpha\colon V\times V \to \FF_2$ like $d$ that, instead of (\ref{eq:cocycle_eqn}), satisfies
\[
	\alpha(v,w)-\alpha(u+v,w)+\alpha(u,v+w)-\alpha(u,v) = f(u,v,w)
\]
for a nonzero \emph{twisting function} $f\colon V\times V\times V \to \FF_2$. We will assume that $\alpha$ satisfies $\alpha(\underline{0},v)=\alpha(v,\underline{0}) = 0$ for all $v\in V$, a property that holds for the cocycle $d$ in the previous section.
The ``twisted cocycle'' terminology comes from geometry, where they appear in other guises; another term used is \emph{factor set}.

From a twisted cocycle $\alpha$ the set $\FF_2 \times V$ can be given a binary operation $\ast_\alpha$:
\begin{equation}\label{eq:loop_extension_prod}
	(s,v)\ast_\alpha(t,w):=(s+ t+ \alpha(v,w),v+w).
\end{equation}
We denote the product $\FF_2 \times V$ of sets equipped with this binary operation by $\FF_2\times_\alpha V$. If the twisting function $f$ is zero, then $\alpha$ is a cocycle, $\ast_\alpha$ is an associative operation, and $\FF_2\times_\alpha V$ is a group.

\begin{definition}
A \emph{loop} is a set $L$ with a binary operation $\star\colon L\times L \to L$, a unit element $e\in L$ such that $e\star x = x \star e = x$ for all $x\in L$, and such that for each $z\in L$, the functions $r_z(x) = x \star z$ and $\ell_z(x)=z\star x$ are bijections $L\to L$. A \emph{homomorphism} of loops is a function preserving multiplication and the unit element.
\end{definition}

The conditions on $r_x$ and $\ell_x$ mean that every element $x\in L$ has a left inverse $x_L^{-1}$ and a right inverse $x_R^{-1}$ for the operation $\star$, and these are unique---but may be different in general. 
The following is a worthwhile exercise using the twisting function and the assumption that $\alpha(\underline{0},v)=\alpha(v,\underline{0})=0$.

\begin{lemma}
The operation $\ast_\alpha$ makes $\FF_2\times_\alpha V$ into a loop, with identity element $(0,\underline{0})$.
The projection function $\pi\colon \FF_2\times_\alpha V \to V$ is a surjective homomorphism of loops, with kernel $\FF_2 \times\{\underline{0}\}$.
\end{lemma}

Groups are examples of loops, but they are, in a sense, the uninteresting case. 
Arbitrary loops are quite badly behaved: even apart from the product being nonassociative, left and right inverses may not coincide, and associativity can fail for iterated products of a single element. 
But there is a special case introduced by Ruth Moufang \cite{Moufang}, with better algebraic properties, while still permitting nonassociativity.

\begin{definition}
A \emph{Moufang loop} is a loop $(L,\star)$ satisfying the identity
\[
x \star (y \star (x \star z)) = ((x \star y) \star x) \star z
\]
for all choices of elements $x,y,z\in L$.
\end{definition}

The most famous example of a Moufang loop is probably the set of non-zero octonions under multiplication, although there are many less-known finite examples. 
A key property of a Moufang loop $L$ is that any subloop $\langle x,y\rangle < L$ generated by a pair of elements $x,y$ is in fact a group. 
By a \emph{subloop} of $L$ we mean a subset containing $e$ that is closed under the operation $\star$.
As a corollary, powers of a \emph{single} element are well-defined, and do not require extra bracketing: $x\star (x \star x) = (x\star x) \star x =: x^3$, for example. 
Additionally, the left and right inverses always agree in a Moufang loop, so that for each $x\in L$ there is a unique $x^{-1}$ such that $x\star x^{-1} = x^{-1}\star x = e$. 
Code loops, defined below as a special case of the construction of $\FF\times_\alpha V$, are examples of Moufang loops.

\begin{example}\label{example:m16}
Let $V= (\FF_2)^3$. The 16-element Moufang loop $M := M_{16}(C_2\times C_4)$ of \cite[Theorem 2]{Chein} is isomorphic to $\FF_2 \times_\mu V$,  where  $\mu\colon V\times V\to\FF_2$  is the twisted cocycle given by the $8\times 8$ array of bits in Figure~\ref{fig:cocycle for M}.
\medskip
\begin{figure}[!ht]
\begin{center}
\includegraphics[width=0.25\textwidth]{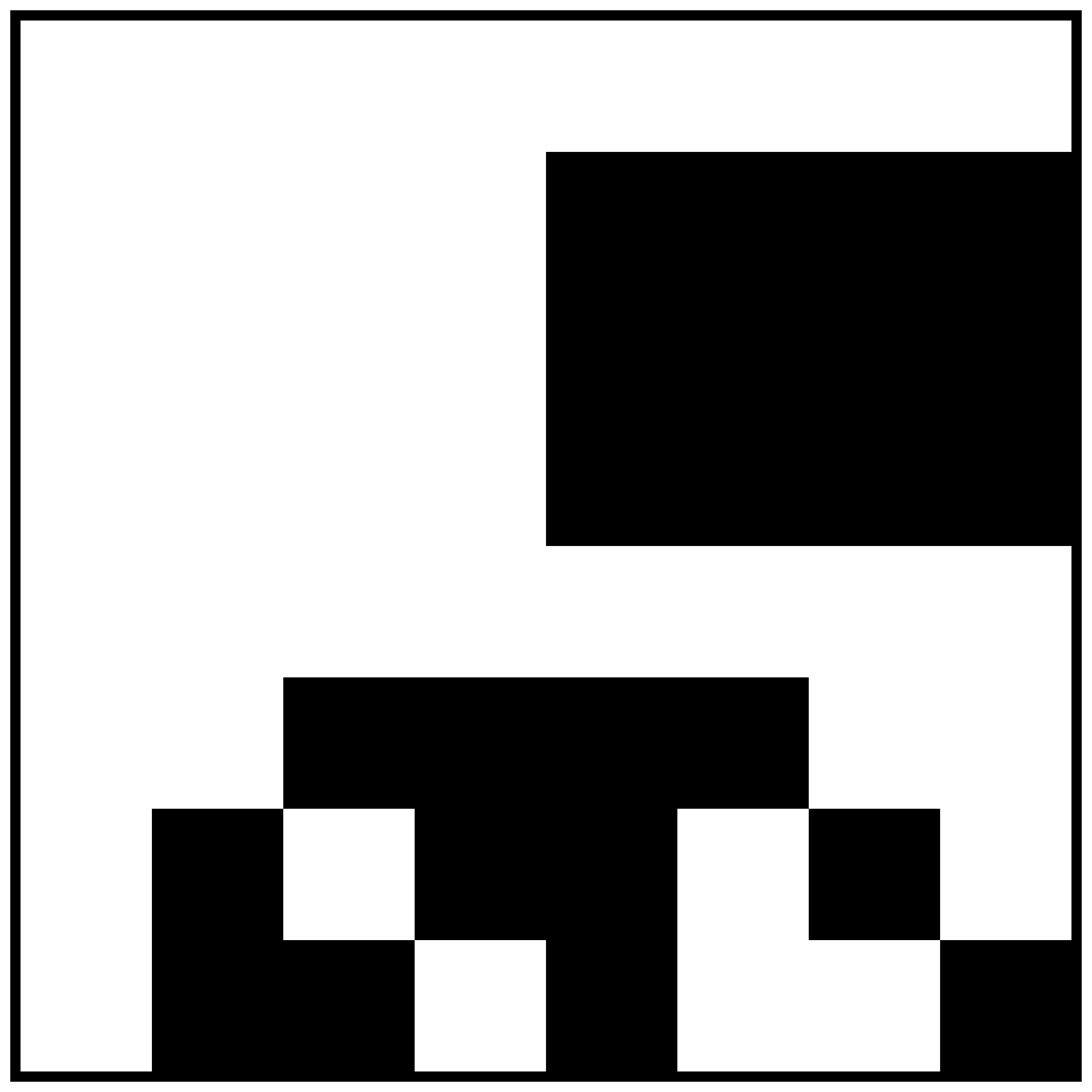}
\end{center}
\caption{Twisted cocycle for the Moufang loop $M$ of Example~\ref{example:m16}; white = $0$, black = $1$. The order of the row/column labels is $000$, $100$, $010$, $110$, $001$, $101$, $011$, $111$. }\label{fig:cocycle for M}
\end{figure}

Notice that the first four columns/rows correspond to the subgroup $U\subset V$ generated by $100$ and $010$, and that the restriction of $\mu$ to $U\times U$ is identically zero (i.e., white). 

This means that the \emph{restriction} $M\big|_U < M$---the subloop of $M$ whose elements are mapped to $U$ by $M \simeq \FF_2\times_\mu V \to V$---is isomorphic to the direct product $\FF_2\times U$, using the definition (\ref{eq:loop_extension_prod}) of the product, and in particular is a group.

\end{example}

\section{Codes and code loops.}

To describe the twisting function $f$ for our code loops, we need to know about some extra operations that exist on vector spaces over the field $\FF_2$. 
For $W$ an $n$-dimensional vector space over $\FF_2$ and vectors $v,w\in W$, there is a vector $v\AND w \in W$ given by
\[
	v\AND w := (v_1w_1,\,v_2w_2,\,\ldots,\,v_nw_n).
\]
If we think of such vectors as binary strings, then this operation is bitwise AND (hence the notation). 
Note that if we take a code $\cC \subset (\FF_2)^n$, then $\cC$ is not guaranteed to be closed under this operation.
The other operation takes a vector $v\in W$ and returns its \emph{weight}: the sum, as an integer, of its entries: $|v| := v_1 + \cdots + v_n$. 
Equivalently, it is the number of nonzero entries in $v$.

The twisting function for a code loop is then a combination of these two, namely $f(u,v,w) := |u\AND v\AND w|$.
However, as alluded to above, we are going to ask that further identities hold. 
For these identities to make sense we need to start with a code with the special property of being \emph{doubly even}.

\begin{definition}
A code $\cC \subset (\FF_2)^n$ is \emph{doubly even} if for every word $v\in \cC$, $|v|$ is divisible by 4. 
\end{definition}

\begin{example}\label{example:Hamming}
The \emph{Hamming (8,4) code} is the subspace $\mathcal{H} \subset (\FF_2)^8$ spanned by the four row vectors
\[
\begin{array}{c}
1\,0\,0\,0\,0\,1\,1\,1\\
0\,1\,0\,0\,1\,0\,1\,1\\
0\,0\,1\,0\,1\,1\,0\,1\\
0\,0\,0\,1\,1\,1\,1\,0
\end{array}
\]
and is doubly even.
\end{example}

A more substantial example of a doubly-even code is the following.

\begin{example}\label{example:Golay}
The (extended binary) Golay code $\cG \subset (\FF_2)^{24}$ is the span of the following (row) vectors, denoted $b_1,\ldots,b_6$ (left) and $b_7,\ldots,b_{12}$ (right):
\[	
	 \begin{array}{c}
     0\,0\,0\,1\,1\,0\,0\,0\,0\,0\,0\,0\,0\,1\,0\,1\,1\,0\,1\,0\,0\,0\,1\,1 \\
     1\,0\,1\,0\,0\,1\,1\,1\,1\,1\,0\,1\,1\,0\,1\,1\,1\,1\,1\,1\,0\,0\,0\,1 \\
     0\,0\,0\,1\,0\,0\,0\,0\,0\,0\,0\,0\,1\,0\,0\,1\,0\,0\,1\,1\,1\,1\,1\,0 \\
     0\,1\,0\,0\,0\,0\,0\,0\,0\,0\,1\,0\,0\,0\,0\,1\,1\,0\,1\,0\,1\,1\,0\,1 \\
     0\,0\,0\,0\,0\,0\,0\,0\,0\,0\,1\,0\,0\,1\,0\,1\,0\,1\,0\,1\,0\,1\,1\,1 \\
     1\,0\,0\,0\,0\,0\,0\,0\,0\,0\,0\,0\,1\,0\,0\,1\,1\,1\,1\,1\,0\,0\,0\,1
     \end{array}
\quad
    \begin{array}{c}
	 1\,0\,1\,0\,0\,1\,0\,1\,1\,1\,0\,0\,1\,1\,1\,0\,0\,1\,1\,1\,1\,1\,1\,1 \\
	 1\,0\,0\,0\,0\,0\,0\,1\,1\,1\,0\,0\,0\,0\,1\,0\,0\,1\,0\,0\,1\,1\,0\,0 \\
	 0\,0\,0\,0\,0\,1\,0\,0\,0\,0\,0\,0\,1\,1\,1\,0\,0\,1\,0\,0\,1\,1\,1\,0 \\
	 1\,0\,0\,0\,0\,0\,0\,0\,1\,0\,0\,0\,1\,1\,1\,0\,0\,0\,1\,1\,1\,0\,0\,0 \\
	 1\,0\,0\,0\,0\,0\,0\,0\,0\,1\,0\,0\,1\,0\,1\,0\,0\,0\,0\,1\,0\,1\,1\,1 \\
	 0\,1\,1\,0\,1\,1\,0\,0\,0\,0\,0\,1\,1\,1\,1\,0\,1\,1\,1\,1\,1\,1\,1\,1
	 \end{array}
\]
This basis is different from the more usual ones (e.g., \cite[Figure 3.4]{SPLAG}), which can be taken as the rows of a $12\times 24$ matrix whose left half is the $12\times 12$ identity matrix.
Our basis, however, allows us to demonstrate some interesting properties.
\end{example}

The inclusion/exclusion formula applied to counting nonzero entries allows us to show that, for all $v$ and $w$ in any doubly even code $\cC$,
\[
	|v+w| + |v\AND w| = |v| + |w| - |v\AND w|\,.
\]
In other words: $|v\AND w| = \frac12(|v| + |w| - |v+w|)$, which implies that $|v\AND w|$ is divisible by $2$.
Thus for words $v,w$ in a doubly even code, both $\frac14|v|$ and $\frac12|v\AND w|$ are integers.

\begin{definition}[Griess \cite{Griess}]
\label{def:codeLoop}
Let $\cC$ be a doubly even code. 
A \emph{code cocycle}\footnote{Griess uses the alternative term ``factor set''.} is a function $\theta\colon \cC \times \cC \to \FF_2$ satisfying the identities
\begin{align}
& \theta(v,w)-\theta(u+v,w)+\theta(u,v+w)-\theta(u,v) =|u\AND v \AND w| \pmod 2; \label{eq: code cocycle 1}\\
& \theta(v,w)+\theta(w,v) = {}  \tfrac12|v\AND w| \pmod 2; \label{eq: code cocycle 2}\\
& \theta(v,v) = {}  \tfrac14|v| \pmod 2.\label{eq: code cocycle 3}
\end{align}
A \emph{code loop} is then a loop arising as $\FF_2\times_\theta \cC$ (up to isomorphism) for some code cocycle $\theta$.
\end{definition}

\begin{example}
Define $\cC := \Span\{b_{10},b_{11},b_{12}\}$, using the basis vectors from Example~\ref{example:Golay}. Then the composite function $\cC\times \cC \simeq (\FF_2)^3\times (\FF_2)^3 \xrightarrow{\mu} \FF_2$, using the twisted cocycle $\mu$ in Figure~\ref{fig:cocycle for M}, is a code cocycle. This makes the loop $M$ from Example~\ref{example:m16} a code loop.
\end{example}

There is a notion of what it means for two twisted cocycles to be equivalent, and equivalent twisted cocycles give isomorphic loops.
As part of \cite[Theorem 10]{Griess}, Griess proves that all code cocycles for a given doubly even code are equivalent, and hence give isomorphic code loops.

\section{Griess's algorithm and its output.}

The algorithm that Griess describes in the proof of \cite[Theorem 10]{Griess} to construct code cocycles for a code $\cC$ takes as input an ordered  basis $\{b_0,\ldots,b_{k-1}\}$ for $\cC$. 
A code cocycle is then built up inductively over larger and larger subspaces $V_i = \Span\{b_0,\ldots,b_i\}$.

However, the description by Griess is more of an outline, using steps like ``determine the cocycle on such-and-such subset using identity X'', where X refers to one of (\ref{eq: code cocycle 1}), (\ref{eq: code cocycle 2}), (\ref{eq: code cocycle 3}), or corollaries of these. We have reconstructed the process in detail in Algorithm~\ref{Griess algo}.

We implemented Algorithm~\ref{Griess algo} in the language Go \cite{RN_GH}, together with diagnostic tests, for instance to verify the Moufang property.
The output of the algorithm is a code cocycle, encoded as a matrix of ones and zeros, and can be displayed as an array of black and white pixels. There are steps where an arbitrary choice of a single bit is allowed; we consistently took this bit to be 0.
For the Golay code this image consists of slightly over 16 million pixels.

As a combinatorial object, a code cocycle $\theta\colon \cG \times \cG\to \FF_2$ constructed from Algorithm~\ref{Griess algo} using the basis in Example~\ref{example:Golay} can be too large and unwieldy to examine for any interesting structure. 
Moreover, to calculate with Parker's loop $\cP := \FF_2\times_\theta \cG$ one needs to know all 16~million or so values of $\theta$.
It is thus desirable to have a method that will calculate values of $\theta$ by a method shorter than Algorithm~\ref{Griess algo}.

\begin{lemma}\label{lemma:formula lemma}
Let $\cC$ be a doubly even code and $\theta$ a code cocycle on it. 
Given $\cC = V\oplus W$ a decomposition into complementary subspaces, then for $v_1,v_2\in V$ and $w_1,w_2\in W$,
\begin{align}\label{eq:theformula}
	\theta(v_1\Plus w_1,v_2\Plus w_2)	
		 =&\ \theta(v_1,v_2)  + \theta(w_1,w_2) + \theta(v_1,w_1) \\
		&+ \theta(w_2,v_2) + \theta(v_1\Plus v_2,w_1\Plus w_2) \nonumber \\
							& + \tfrac12|v_2\AND(w_1\Plus w_2)| + |v_1\AND v_2 \AND (w_1\Plus w_2)| \nonumber \\
							&+|w_1\AND w_2 \AND v_2| + \left|v_1\AND w_1 \AND (v_2 \Plus  w_2)\right| \pmod 2\,. \nonumber
\end{align}
\end{lemma}

\begin{proof}
We apply the identity (\ref{eq: code cocycle 1}) three times and then the identity (\ref{eq: code cocycle 2}) once:
\begin{align*}
\begin{split}
	\theta(v_1+w_1,v_2+w_2) =&\ \theta(v_1,w_1) + \theta(v_1,v_2+w_1+w_2) + \theta(w_1,v_2+w_2) \\
	&+ |v_1\AND w_1\AND(v_2+w_2)| 
\end{split}\\
\begin{split}
	 =&\ \theta(v_1,w_1) + \big\{ \theta(v_1,v_2) + \theta(v_1+v_2,w_1+w_2)) \\ 
	 & + \theta(v_2,w_1+w_2) + |v_1\AND v_2 \AND (w_1+w_2)| \big\} \\
	 & + \big\{\theta(w_1,w_2) + \theta(w_1+w_2,v_2) + \theta(w_2,v_2) \\
	 & + |w_1\AND w_2\AND v_2| \big\} 
	  + |v_1\AND w_1\AND(v_2+w_2)| 
\end{split}\\
\begin{split}
	 =&\ \theta(v_1,w_1) + \big\{ \theta(v_1,v_2) + \theta(v_1+v_2,w_1+w_2)) \\ 
	 & +  |v_1\AND v_2 \AND (w_1+w_2)| \big\} \\
	 & + \big\{\theta(w_1,w_2) + \theta(w_2,v_2) + |w_1\AND w_2\AND v_2| \big\} \\
	 & + |v_1\AND w_1\AND(v_2+w_2)| + \tfrac12|v_2\AND (w_1+w_2)|. \qedhere
\end{split}
\end{align*}
\end{proof}

Observe that in Lemma~\ref{lemma:formula lemma}, on the right-hand side of (\ref{eq:theformula}), the code cocycle $\theta$ is only ever evaluated on vectors from the subset $V \cup W \subset \cC$.
This means we can throw away the rest of the array and still reconstruct arbitrary values of $\theta$ using (\ref{eq:theformula}).
If we assume that $\cC$ is $2k$-dimensional, and that $V$ and $W$ are both $k$-dimensional, then the domain of the restricted $\theta$ has $(2^k+2^k - 1)^2 = 2^{2(k+1)} - 2^{k+2} + 1 = O((2^k)^2)$ elements. 
Compared to the full domain of $\theta$, which has $2^{2k}\times 2^{2k} = (2^k)^4$ elements, this is roughly a square-root saving.

However, one heuristic for choosing the subspaces $V$ and $W$ is to aim for a reduction in the apparent randomness of the plot of the restricted code cocycle, or equivalently, less-granular structural repetition. 
This is true, even if the size of $V\cup W$ is not minimized by choosing $V$ and $W$ to have dimension $(\dim\cC)/2$ (or as close as possible).

Now it should be clear why the Golay code basis in Example~\ref{example:Golay} was partitioned into two lists of six vectors: we can reconstruct all $16\,777\,216$ values of the resulting code cycle $\theta$, and hence the multiplication in Parker's loop, from a mere $2^{14} - 2^8 + 1 = 16\,129$ values.
The span of the left column of vectors in Example~\ref{example:Golay} is the subspace $V\subset \cG$, and the span of the right column of vectors is $W\subset \cG$.

The top left quadrant of Figure~\ref{fig:Parker cocycle} then contains the restriction of $\theta$ to $V\times V$, and the bottom right quadrant the restriction to $W\times W$. 
The off-diagonal quadrants contain the values of $\theta$ restricted to $V\times W$ and $W\times V$.
\begin{figure}[!b]
\begin{center}
\includegraphics[width=0.7\textwidth]{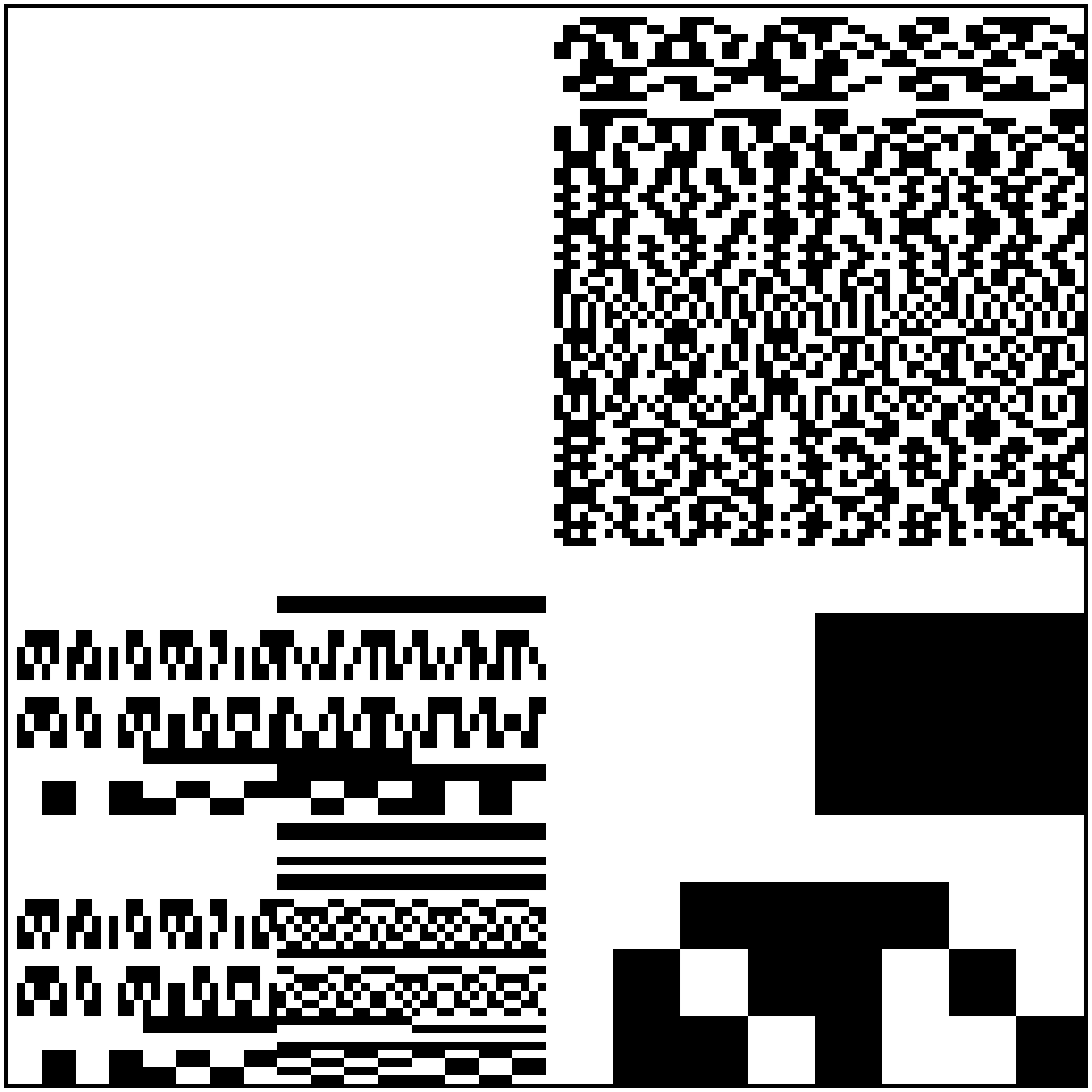}
\end{center}
\caption{The restriction of the code cocycle $\theta$ for Parker's loop to $(V\cup W)^2$. A machine-readable version is available in \cite{RN_GH} or as an arXiv ancillary file. The order of the row/column labels is $\underline{0},b_1,b_2,b_1+b_2,b_3,\ldots, b_1+\cdots +b_6$, $b_7,b_8,b_7+b_8,\ldots,b_7+\cdots + b_{12}$.}
\label{fig:Parker cocycle}
\end{figure}
From Figure~\ref{fig:Parker cocycle} and formula (\ref{eq:loop_extension_prod}) we can see that the restriction $\cP\big|_V$ of Parker's loop is a direct product (hence is a group), as $\theta\big|_{V\times V}$ is identically zero. 
Moreover, the restriction of $\cP\big|_W$ is isomorphic to the direct product $(\FF_2)^3 \times M$, where $M$ is the Moufang loop from Example~\ref{example:m16}.
This is because what was a single pixel in Figure~\ref{fig:cocycle for M} is now an $8\times 8$ block of pixels in the lower right quadrant of Figure~\ref{fig:Parker cocycle}.

\section{Discussion and comparison.}\label{sec:discussion}

Parker's loop $\cP$ is an algebraic structure that has spawned a small industry trying to understand and construct code loops in general, as they form a class of Moufang loops that are relatively easy to describe.
Parker's original treatment of this class of Moufang loops used vector spaces $V$ over $\FF_2$ equipped with a function $V \to \FF_2$ satisfying certain conditions analogous to those in Definition~\ref{def:codeLoop} (see \cite[Definition~13]{Griess}).

Aside from Parker's approach and the description by Griess using code cocycles, there was an unpublished thesis \cite{Johnson}, characterizations as loops with specified commutators and associators \cite{CheinGoodaire}, an iterative construction using centrally twisted products \cite{Hsu}, and a construction using groups with triality \cite{Nagy}. 
The LOOPS library \cite{loops3.4.1} for the software package GAP \cite{GAP4} contains all the code loops of order $64$ and below, although ``the package is intended primarily for quasigroups and loops of small order, say up to $1000$''.
Even the more recent \cite{OBrien_Vojtechovsky}, which classifies code loops up to order $512$ in order to add them to the LOOPS package, falls short of the $8192$ elements of Parker's loop; the authors say ``our work suggests that it will be difficult to extend the classification of code loops beyond order $512$''.
In principle, there is nothing stopping the construction of $\cP$ in LOOPS, but it will essentially be stored as a multiplication table, which would comprise $67\,108\,864$ entries, each of which is a $13$-bit element label.
The paper \cite{Morier-Genoud_Ovsienko} describes an algebraic formula for a code cocycle that will build Parker's loop, as a combination of the recipe in the proof of Proposition~6.6 and the generating function in Proposition~9.2 appearing there.
This formula is a polynomial with $330$ cubic terms and $12$ quadratic terms in $24$ variables, being coefficients of basis vectors of $\cG$. 
To compare, combined with the small amount of data in Figure~\ref{fig:Parker cocycle} together with a labelling of rows/columns by words of the Golay code, Lemma~\ref{lemma:formula lemma} only requires eight terms, of which four are cubic and the rest come from the $16\,129$ stored values of $\theta$ (the term $\theta(v_1,v_2)$ vanishes identically, for $\cP$).
Since large-scale computation in large code loops (for instance in a package like LOOPS) requires optimizing the binary operation, we have found a space/time trade-off that vastly improves on existing approaches.

In addition to computational savings, the ability to visually explore the structure of code loops during experimentation more generally is a novel advance---the recognition of $(\FF_2)^3\times M$ inside $\cP$ was purely by inspection of the picture of the code cocyle then consulting the (short) list of Moufang loops of small order in \cite{Chein}.
Discovery of the basis in Example~\ref{example:Golay} occurred by walking through the spaces of bases of subcodes and working with the heuristic that more regularity in the appearance of the code cocycle is better.
Additionally, our software flagged when subloops thus considered were in fact associative, and hence a group, leading to the discovery of the relatively large elementary subgroups $(\FF_2)^7 < \cP$ and $(\FF_2)^6 < (\FF_2)^3\times M < \cP$. 

Finally, one can also remark that because of the identity (\ref{eq: code cocycle 2}), one can replace the formula in Lemma~\ref{lemma:formula lemma} by one that is only ever evaluated on $V^2$, $W^2$, or $V\times W$, with just one more term. This allows the reconstruction the top right quadrant of Figure~\ref{fig:Parker cocycle} from the bottom left quadrant. 
Thus one can describe Parker's loop as being generated by the subloops $(\FF_2)^7$ and $(\FF_2)^3\times M$, with relations coming from the information contained in the bottom left quadrant of Figure~\ref{fig:Parker cocycle}, and the formulas (\ref{eq: code cocycle 2}) and (\ref{eq:theformula}).
The apparent structure in that bottom left quadrant is intriguing, and perhaps indicative of further simplifications; this will be left to future work.

\begin{algorithm}[!hb]
\small
\caption{Reverse-engineered from proof of \cite[Theorem 10]{Griess}.}\label{Griess algo}
\SetAlgoLined
\DontPrintSemicolon
\KwData{Basis $B = \{b_0,b_1,\ldots,b_{k-1}\}$ for the code $\cC$}
\KwResult{Code cocycle $\theta\colon \cC\times \cC\to \FF_2$, encoded as a square array of elements from $\FF_2$, with rows and columns indexed by $\cC$ }
\;
\tcp{Initialise} 
\ForAll{$c_1,c_2 \in C$}{
$\theta(c_1,c_2) \leftarrow 0$\;
}

$\theta(b_0,b_0) \leftarrow \tfrac14\left|b_0\right|$\;
\;
\ForAll{$1\leq i\leq \Length(B)$}{
	Define $V_i :=\Span\{b_0,\ldots,b_{i-1}\}$\;
	\tcp{(D1) define theta on \{bi\} x Vi then deduce on Vi x \{bi\}}
	\ForAll{$v\in V_i$}{
		\eIf{$v\neq0$}{
			$\theta(b_i,v) \leftarrow \text{random}$ \tcp*{In practice, random = 0}
			$\theta(v,b_i) \leftarrow \tfrac12\left|v\AND b_i\right|+\theta(b_i,v)$\;
		}
		{
			\tcp{$\theta(b_i,v)$ is already set to 0}
			$\theta(v,b_i)\leftarrow \tfrac12\left|v\AND b_i\right|$\;
		}
	}
	\tcp{(D2) deduce theta on \{bi\} x Wi and Wi x \{bi\}}
	\ForAll{$v\in V_i$}{
		$\theta(b_i,b_i\Plus v) \leftarrow \tfrac14\left|b_i\right| + \theta(b_i,v)$\;
		$\theta(b_i\Plus v,b_i) \leftarrow \tfrac12\left|b_i\AND (b_i\Plus v)\right| + \tfrac14\left|b_i\right| + \theta(b_i,v)$\;
	}
	\tcp{(D3) deduce theta on Wi x Wi}
	\ForAll{$v_1\in V_i$}{
		\ForAll{$v_2\in V_i$}{
			$w\leftarrow b_i \Plus v_2$\;
			$a\leftarrow \theta(v_1,b_i)$\;
			$b\leftarrow \theta(v_1,b_i\Plus w)$\;
			$c\leftarrow \theta(w,b_i)$\;
			$r \leftarrow \tfrac12\left|v_1\AND w\right| + a + b + c$\;
			$\theta(w,b_i\Plus v_1) \leftarrow r$
		}
	}
	\tcp{(D4) deduce theta on Wi x Vi and Vi x Wi}
	\ForAll{$v_1\in V_i$}{
		\ForAll{$v_2\in V_i$}{
			$w\leftarrow b_i \Plus v_2$\;
			$a\leftarrow \theta(w,v_1\Plus w)$\;
			$\theta(w,v_1) \leftarrow \tfrac14\left|w\right| + a$\;
			$\theta(v_1,w) \leftarrow \tfrac12\left|v_1\AND w\right| + \tfrac14\left|w\right| + a$\;
		}
	}
}
\end{algorithm}

\section{Acknowledgments.}
DMR is supported by the Australian Research Council's Discovery Projects scheme (project number DP180100383), funded by the Australian Government.
The authors thank an anonymous referee for providing Coxeter's example and clarifying the details of Parker's construction.

\bibliographystyle{amsalpha}
\bibliography{codeloops}

\providecommand{\bysame}{\leavevmode\hbox to3em{\hrulefill}\thinspace}
\providecommand{\MR}{\relax\ifhmode\unskip\space\fi MR }
\providecommand{\MRhref}[2]{%
  \href{http://www.ams.org/mathscinet-getitem?mr=#1}{#2}
}
\providecommand{\href}[2]{#2}
\begin{thebibliography}{MGO11}

\bibitem[Bro94]{Brown}
Kenneth Brown, \emph{Cohomology of groups}, Graduate Texts in Mathematics,
  vol.~87, Springer-Verlag, New York, 1994, Corrected reprint of the 1982
  original.

\bibitem[CG90]{CheinGoodaire}
Orin Chein and Edgar~G. Goodaire, \emph{Moufang loops with a unique nonidentity
  commutator (associator, square)}, J. Algebra \textbf{130} (1990), no.~2,
  369--384.

\bibitem[Che74]{Chein}
Orin Chein, \emph{Moufang loops of small order. {I}}, Trans. Amer. Math. Soc.
  \textbf{188} (1974), 31--51.

\bibitem[Con85]{Conway}
John.~H. Conway, \emph{A simple construction for the {F}ischer-{G}riess monster
  group}, Invent. Math. \textbf{79} (1985), no.~3, 513--540.

\bibitem[Cox46]{Coxeter}
H.~S.~M. Coxeter, \emph{Integral {C}ayley numbers}, Duke Math. J. \textbf{13}
  (1946), 561--578.

\bibitem[CS99]{SPLAG}
John~H. Conway and Neil J.~A. Sloane, \emph{Sphere {P}ackings, {L}attices and
  {G}roups}, third ed., Grundlehren der Mathematischen Wissenschaften, vol.
  290, Springer-Verlag, New York, 1999.

\bibitem[GAP19]{GAP4}
The GAP~Group, \emph{{GAP -- Groups, Algorithms, and Programming, Version
  4.10.2}}, 2019, \url{www.gap-system.org}.

\bibitem[Gri86]{Griess}
Robert~L. Griess, Jr., \emph{Code loops}, J. Algebra \textbf{100} (1986),
  no.~1, 224--234.

\bibitem[Gri87]{Griess87}
\bysame, \emph{Sporadic groups, code loops and nonvanishing cohomology}, J.
  Pure Appl. Algebra \textbf{44} (1987), no.~1-3, 191--214.

\bibitem[Hsu00]{Hsu}
Tim Hsu, \emph{Explicit constructions of code loops as centrally twisted
  products}, Math. Proc. Cambridge Philos. Soc. \textbf{128} (2000), no.~2,
  223--232.

\bibitem[Joh88]{Johnson}
Peter~Malcolm Johnson, \emph{Gamma spaces and loops of nilpotence class two},
  Ph.D. thesis, University of Illinois at Chicago, 1988.

\bibitem[MGO11]{Morier-Genoud_Ovsienko}
Sophie Morier-Genoud and Valentin Ovsienko, \emph{A series of algebras
  generalizing the octonions and {H}urwitz-{R}adon identity}, Comm. Math. Phys.
  \textbf{306} (2011), no.~1, 83--118.

\bibitem[Mou35]{Moufang}
Ruth Moufang, \emph{Zur {S}truktur von {A}lternativk\"{o}rpern}, Math. Ann.
  \textbf{110} (1935), no.~1, 416--430.

\bibitem[Nag08]{Nagy}
G\'{a}bor~P. Nagy, \emph{Direct construction of code loops}, Discrete Math.
  \textbf{308} (2008), no.~23, 5349--5357.

\bibitem[NR19]{RN_GH}
Ben Nagy and David~Michael Roberts, \emph{{\textrm{\texttt{codeloops}}}
  library}, 2019, \url{github.com/bnagy/codeloops}.

\bibitem[NV18]{loops3.4.1}
G\'{a}bor~P. Nagy and Petr Vojt{\v e}chovsk{\a'y}, \emph{{loops}, computing
  with quasigroups and loops in gap, {V}ersion 3.4.1}, \href
  {gap-packages.github.io/loops/}{\texttt{gap-packages.github.io/}\discretionary
  {}{}{}\texttt{loops/}}, Nov 2018, Refereed GAP package.

\bibitem[OV17]{OBrien_Vojtechovsky}
Eamonn~A. O'Brien and Petr Vojt\v{e}chovsk\'{y}, \emph{Code loops in dimension
  at most 8}, J. Algebra \textbf{473} (2017), 607--626.

\end{thebibliography}
\end{document}